\documentclass[12pt]{amsart}
\usepackage[mathscr]{eucal}
\usepackage{amssymb}
\usepackage{amsfonts}
\usepackage{latexsym}
\usepackage{amsthm}
\usepackage{graphicx}
\usepackage{enumerate}
\usepackage{xcolor}

\theoremstyle{plain}

\newtheorem{tw}{Theorem}[section]
\newtheorem{lem}[tw]{Lemma}
\newtheorem{wn}[tw]{Corollary}
\newtheorem{fakt}[tw]{Fact}
\newtheorem{prop}[tw]{Proposition}

\newtheorem{question}[tw]{Question}

\theoremstyle{definition}
\newtheorem{df}[tw]{Definition}
\newtheorem{remark}[tw]{Remark}

\newcommand{\bumpineq}{\prec}
\newcommand{\N}{\mathbb{N}}
\newcommand{\I}{\mathbf{I}}
\newcommand{\RR}{\mathbb{R}}

\DeclareMathOperator{\cl}{cl}
\DeclareMathOperator{\bd}{bd}

\DeclareMathOperator{\supp}{supp}

\begin{document}
\title[On uniformly continuous maps between function spaces]{On uniformly continuous maps between function spaces}
\author[R.\ G\'orak]{Rafa\l ~G\'orak}
\address{Technical University of Warsaw\\ Pl. Politechniki 1\\ 00--661 Warszawa, Poland}
\email{R.Gorak@mini.pw.edu.pl}
\author[M.\ Krupski]{Miko\l aj Krupski}
\address{Department of Mathematics\\ University of Pittsburgh\\
Pittsburgh, PA 15260, USA\\
and \\
Institute of Mathematics\\ University of Warsaw\\ ul. Banacha 2\\
02--097 Warszawa, Poland }
\email{m.krupski@pitt.edu}
\author[W.\ Marciszewski]{Witold Marciszewski}
\address{Institute of Mathematics\\ University of Warsaw\\ ul. Banacha 2\\
02--097 Warszawa, Poland }
\email{wmarcisz@mimuw.edu.pl}
%\date{\today}

\subjclass[2010]{54C35}%
\keywords{Function spaces; Pointwise topology; Uniform
homeomorphisms;
$C_p(X)$ space}%

%\thispagestyle{empty}
%\setcounter{page}{0}
%\newpage
\begin{abstract}
In this paper we develop a technique of constructing uniformly continuous maps between function spaces $C_p(X)$ endowed with the pointwise topology.
%The technique was first introduced by Gulko who proved that all the countable compacta are $u$-equivalent.
%Using this technique we obtain main results of this paper. We prove that if a space $X$ is compact,
We prove that if a space $X$ is compact
metrizable and strongly countable-dimensional, then there exists a uniformly continuous surjection from $C_p([0,1])$ onto $C_p(X)$.
We provide a partial result concerning the reverse implication. We also show that, for every infinite Polish zero-dimensional space $X$,
the spaces $C_p(X)$ and $C_p(X) \times C_p(X)$ are uniformly homeomorphic. This partially answers two questions posed by Krupski and Marciszewski.
\end{abstract}
\maketitle
\section{Introduction}

For a Tychonoff space $X$, by $C_p(X)$ we denote the space of all continuous real-valued functions on $X$ endowed with the pointwise topology.

This paper contributes to the study of uniformly continuous maps between function spaces. Let us recall that that a map $\varphi:C_p(X)\to C_p(Y)$
is \textit{uniformly continuous} if for each open neighborhood $U$ of the zero function in $C_p(Y)$, there is and open neighborhood $V$ of the zero function
in $C_p(X)$ such that $(f-g)\in V$ implies $(\varphi(f)-\varphi(g))\in U$. Building on a work of Gul'ko \cite{Gu1} and G\'orak \cite{Gor} we
further develop a technique of constructing uniformly continuous maps between function spaces. We shall apply this technique in two rather different contexts:
\begin{enumerate}[1.]
 \item To show that for any compact
metrizable and strongly countable-dimensional space $X$, there  is a uniformly continuous surjection from $C_p([0,1])$ onto $C_p(X)$
 \item To prove that for every infinite Polish zero-dimensional space $X$, the spaces $C_p(X)$ and $C_p(X) \times C_p(X)$ are uniformly homeomorphic.
\end{enumerate}
It seems in place to present here some motivations of our work.

After Pestov \cite{Pestov} proved that $\dim X=\dim Y$ provided $C_p(X)$ and $C_p(Y)$ are linearly homeomorphic a natural problem arose whether the dimension of a
space can be raised by continuous linear surjections of function spaces. Leiderman, Morris and Pestov \cite{LMP}
answered this question by
showing that for any compact metrizable finite-dimensional space $X$, the space $C_p(X)$ is a linear continuous image of $C_p([0,1])$ (this result was
strengthened by Levin who proved in \cite{Levin} that under the same assumptions $C_p(X)$ is even an open linear continuous image of $C_p([0,1])$).
Levin and Sternfeld noticed that there are also infinite-dimensional spaces $X$, for which $C_p(X)$ is a linear continuous image of $C_p([0,1])$
(see \cite[Remark 4.6]{LMP}) and the problem of characterizing spaces $X$ for which $C_p(X)$ is a linear continuous image of $C_p([0,1])$ remains open
(see \cite[Problem 4.4]{Levin}). Of course, the assumption that $C_p(X)$ is a linear continuous image of $C_p([0,1])$ imposes some restrictions on the space $X$.
Namely,
$X$ has to be compact metrizable and
strongly countable-dimensional
\footnote{A normal space $X$ is \textit{strongly countable-dimensional} if $X$ can be represented as a countable union of closed finite-dimensional subspaces.}
(see, e.g. \cite{GartFeng}).
As observed by Gartside and Feng \cite{GartFeng} strongly countable-dimensional metrizable compacta stratify by $fd$-height
(see Definition \ref{def_fd_pochodna} below): To every strongly countable-dimensional metrizable compactum $X$ one can assign a countable ordinal
(the $fd$-height of $X$) which is equal to 1 if $X$ is finite-dimensional. Gartside and Feng improved Leiderman-Morris-Pestov theorem by showing that
$C_p(X)$ is a linear continuous image of $C_p([0,1])$ if $X$ is compact metrizable and has finite $fd$-height (see \cite[Theorem 1]{GartFeng}). Their proof proceeds
by induction on $fd$-height. However, certain constants involved in the inductive construction, that control norms of linear maps, prevent it
going transfinitely.

In this paper we show that this problem disappears for uniformly continuous surjections, i.e. for any compact metrizable strongly countable-dimensional $X$,
the space
$C_p(X)$ is a uniformly continuous image of $C_p([0,1])$. We also provide a partial result concerning the reverse implication (see Theorem
\ref{presrvation_count_dim} below).

The problem whether, for ``sufficiently nice'' $X$, the space $C_p(X)$ is (linearly/uniformly) homeomorphic to its square $C_p(X)\times C_p(X)$ has a long history
and, usually, it is not easy to settle, cf. \cite[page 361]{Marciszewski-artykul-przegladowy}, \cite{KM}.
It was shown by Arhangel'skii \cite{A2}, Baars and de Groot \cite{BG} that for an infinite Polish zero-dimensional space $X$ which is either compact
or not $\sigma$-compact, the space $C_p(X)$ is linearly homeomorphic to $C_p(X)\times C_p(X)$. Thus, a natural question arises
(see \cite[Question 5.8]{KM}): \textit{Suppose that $X$ is an infinite
Polish zero-dimensional $\sigma$-compact space. Is it true that $C_p(X)$ is (linearly) homeomorphic to $C_p(X)\times C_p(X)$?}
We shall give the affirmative answer for uniform homeomorphisms, see Theorem \ref{tw_kwadrat} below (the question for linear
homeomorphisms remains open). As a corollary we get a partial (affirmative) answer to \cite[Question 5.7]{KM}, i.e. we show that
for an infinite countable metrizable space $X$, the spaces $C_p(X)$ and $C_p(X) \times C_p(X)$ are uniformly homeomorphic.

The paper is organized as follows. Section 2 introduces basic notation and contains some
auxiliary results whose proofs are either straightforward or are known - the interested reader should consult the papers
\cite{GartFeng} and \cite{Gor}. In Section 3 we show that if $X$ is a compact metrizable space and
$C_p(Y)$ is a uniformly continuous image of $C_p(X)$ then $Y$ is compact metrizable too. This result was known before and easily follows from a theorem of
Uspenskii \cite[Proposition 2]{U}. However, the proof which we give is new. %and shows certain similarity between uniform and linear maps.
In Section 4 we discuss the following problem: \textit{Suppose that $X$ is compact metrizable and strongly countable-dimensional. Is it true that $Y$
is strongly countable-dimensional provided $C_p(Y)$ is a uniformly continuous image of $C_p(X)$?} We prove that the answer is ``yes'' if a uniformly continuous
map satisfies an additional condition. Section 5 contains a proof of one of the main results of the paper which in particular says that for any compact metrizable
strongly countable-dimensional space $X$, the space $C_p(X)$ is a uniformly continuous image of $C_p([0,1])$.
The last, Section 6 is devoted to the study of uniform homeomorphisms between $C_p(X)$ and its square $C_p(X)\times C_p(X)$, for an infinite Polish zero-dimensional
space $X$.

%For a Tychonoff space $X$, by $C_p(X)$ we denote the space of all continuous real-valued functions on $X$ endowed with the pointwise topology.
%Extensive work in the theory of $C_p(X)$-spaces has been done on various types of continuous maps between function spaces
%One of the basic lines of research in the theory of $C_p(X)$-spaces is the study of various types of continuous maps between function spaces, and
%the central problem

%one of the central problems is the theory of function space is the question
%An interesting and important line of research in the theory of $C_p(X)$-spaces is to determine whether there is a map: a (linear) continuous surjection,
%a (linear) homeomorphism, etc. between two function spaces.

%\begin{df}
%The  space $X$ $u$-dominates $Y$  if there exists a uniformly continuous surjection $\varphi:C_p(X) \mapsto C_p(Y)$.
%\end{df}
\section{Preliminary facts and definitions}
We denote the unit interval $[0,1]$ by $\I$. We say that a space $Y$ is
{$u$-dominated} by $X$, if there exists a uniformly continuous surjection $\varphi:C_p(X) \to C_p(Y)$.

Since in some parts of our paper we deal with $C_p(X)$ spaces where $X$ is not compact, we use the notion of extended norm, cf. \cite{Be}.
\begin{df}
The extended norm on a linear spaces $E$ is a function $\|.\|:E \to [0,\infty]$ that satisfies properties of the conventional norm:
\begin{itemize}
\item[(i)] $\forall x \in E$ $\|x\|=0$ iff $x=0$;
\item[(ii)] $\forall x \in E\ \forall a \in \RR$ $\|ax\|=|a|\|x\|$;
\item[(iii)] $\forall x,y \in E$ $\|x+y\| \leq \|x\|+\|y\|$.
\end{itemize}
In our considerations we set $0 \cdot \infty=0$.
\end{df}
We will use the following slight modification of the relation introduced in \cite{Gor} which was motivated by the work of Gul'ko \cite{Gu1}:
\begin{df}
\label{def_rownosci} Let $E$ and $F$ be linear topological spaces
and  $\| \cdot \|_1$, $\| \cdot \|_2$ be extended norms, on $E$ and $F$,
respectively, not necessarily related to the topologies. We write
$(E,\| \cdot \|_1) \bumpeq (F,\| \cdot \|_2)$ if, for every
$\varepsilon
>0$, there exists a uniform homeomorphism $u_\varepsilon : E
\rightarrow F$ satisfying the following condition:
$$
 (a_\varepsilon) \;\;\; (1 + \varepsilon)^{-1}\|f\|_1
 \leq \|u_\varepsilon(f)\|_2 \leq \|f\|_1 \textrm{ for every }
 f \in E.
$$
If it is clear which extended norms are considered on $E$ and $F$ we write
$E \bumpeq F$.
\end{df}
The next two definitions were motivated by the idea of $c$-good maps introduced in \cite{GartFeng}.

\begin{df}\label{c-good} Let $(E,\| \cdot \|_1)$ and $(F,\| \cdot \|_2)$ be normed spaces and let $c$ be a positive number. We say that the map $\Phi: E\to F$ is $c$-good if
$$\forall f \in F\ \exists e \in E\ \ \Phi(e)=f \mbox{ and } \|e\|_1 \leq c\|f\|_2.$$
\end{df}

Observe that, by the Closed Graph Theorem, for compact spaces $X$ and $Y$, every linear continuous surjection $\varphi: C_p(X)\to C_p(Y)$ is $c$-good for some $c>0$.

\begin{df}
\label{def_nierownosci} Let $E$ and $F$ be linear topological spaces
and  $\| \cdot \|_1$, $\| \cdot \|_2$ be norms, on $E$ and $F$,
respectively, not necessarily related to the topologies. We write
$(F,\| \cdot \|_2) \bumpineq (E,\| \cdot \|_1)$ if, for every
$\varepsilon
>0$, there exists a uniform surjection $u_\varepsilon : E
\rightarrow F$ satisfying the following conditions:
\begin{itemize}
\item[(i)] $\forall e \in E$ $\|u_\varepsilon(e)\|_2 \leq \|e\|_1$,

\item[(ii)] $u_\varepsilon$ is $(1+\varepsilon)$-good.
\end{itemize}
If it is clear which norms are considered on $E$ and $F$ we write
$F \bumpineq E$.
\end{df}

Let us remark that in this paper we do not need to extend the above definition for extended norms.

\begin{df}
\label{def_produktu}
Let $\{E_i: i \in I\}$ be a collection of linear topological spaces
and let  $\| \cdot \|_{E_i}$ be extended norm on $E_i$,
 not necessarily related to the topology.
In the product $\Pi_{i \in I}E_i$ we will always consider the standard product topology and the extended norm \mbox{$\|\cdot\|:\Pi_{i \in I}E_i \to [0,\infty]$} given by:
$$\|(f_i)_{i \in I}\| =\sup_{i \in I} \|f_i\|_{E_i}\,.$$
\end{df}

\begin{df}
Let $(E_i)_{i \in \N}$, be a sequence of linear topological spaces with extended norms  $\| \cdot \|_{E_i}$, not necessarily related to the topologies.
By $\Pi^*_{i \in \N} E_i$ we denote the $c_0$-product of spaces $E_i$, i.e., a subspace of $\Pi_{i \in \N}E_i$ consisting of sequences $(f_i)_{i \in \N}$ such that $\lim_{i \to \infty}\|f_i\|_{E_i}=0$. The norm and the topology on $\Pi^*_{i \in \N} E_i$ is inherited from $\Pi_{i \in \N}E_i$ as declared in Definition \ref{def_produktu}.
\end{df}

The following simple facts will be used extensively throughout the whole paper:
\begin{fakt}
\label{nierownosci} Let $E$ and $F$ be linear topological spaces
and  $\| \cdot \|_1$, $\| \cdot \|_2$ be norms, on $E$ and $F$,
respectively, not necessarily related to the topologies.
If $E \bumpeq F$ then $E \bumpineq F$ and $F \bumpineq E$.
\end{fakt}

\begin{fakt}\label{iloczyny}
Let $E_i$ and $F_i$, $i \in I$, be linear topological spaces
and  $\| \cdot \|_{E_i}$, $\| \cdot \|_{F_i}$ be extended norms, on $E_i$ and $F_i$,
respectively, not necessarily related to the topologies. Then
$$\forall i \in I \; E_i \bumpeq F_i \Rightarrow \Pi_{i \in I } E_i \bumpeq \Pi_{i \in I } F_i.$$
\end{fakt}

\begin{fakt}\label{c_0_iloczyny}
Let $E_i$ and $F_i$, $i \in \N$, be linear topological spaces
and  $\| \cdot \|_{E_i}$, $\| \cdot \|_{F_i}$ be extended norms, on $E_i$ and $F_i$,
respectively, not necessarily related to the topologies.
Then $$\forall i \in \N \; E_i \bumpeq F_i \Rightarrow \Pi^*_{i \in \N } E_i \bumpeq \Pi^*_{i \in \N } F_i.$$
Moreover, if $\| \cdot \|_{E_i}$ and $\| \cdot \|_{F_i}$, $i \in \N$, are usual norms then
$$\forall i \in \N \; E_i \bumpineq F_i \Rightarrow \Pi^*_{i \in \N } E_i \bumpineq \Pi^*_{i \in \N } F_i.$$
\end{fakt}
  
On all subspaces of $C_p(X)$ ($X$ is not necessarily compact) we will always consider the supremum extended norm. If $A$ is a closed subspace of space $X$, then we denote the subspace $\{f\in C_p(X): f\upharpoonright A\equiv 0\}$ by $C_p(X,A)$.

While the above facts are easy to verify the following ones are more challenging.
\begin{prop}[{\cite[Proposition 2.13]{Gor}} for $X$ compact]
\label{wn_Dug} Let A, B be closed subsets of a metrizable space X such that
$B \subset A$. Then $$C_p(X,B) \bumpeq C_p(A,B) \times C_p(X, A).$$
\end{prop}
Although the proof of the above fact is presented in \cite{Gor}, for $X$ compact, the same proof works also for all metrizable spaces $X$ since the the Dugundji extension theorem holds for all such spaces.
Applying this proposition one can easily obtain:

\begin{wn}\label{wn_wn_Dug}
For every closed subset $A$ of a metrizable space $X$ we have $C_p(X) \bumpeq
C_p(X,A) \times C_p(A)$. Consequently, $C_p(X,A)\bumpineq C_p(X)$ and $C_p(A) \bumpineq C_p(X)$ for $X$ compact.
\end{wn}
Repeating the argument from the proof of Corollary 2.15 in \cite{Gor} one can show:
\begin{wn}
\label{o_punkcie} $C_p(X,x_0) \bumpeq C_p(X)$ for every $x_0 \in
X$, where $X$ is nondiscrete and metrizable.
\end{wn}
We will also use another result from \cite{Gor}:
\begin{fakt}{\cite[Fact 2.16]{Gor}}\label{prod_Cp(I)}
\label{lem_prod*_I}
$C_p(\I) \bumpeq \Pi^*_{i \in \N }
C_p(\I)$.
\end{fakt}

Another important fact concerning $\bumpineq$ was proved in \cite{GartFeng}.

\begin{fakt}{\cite[Theorem 6]{GartFeng}}\label{dominacja_kostki}
For every $n \in \mathbb{N}$, we have $C_p(\I^n) \bumpineq C_p(\I)$ and hence $C_p(K) \bumpineq C_p(\I)$ for every finite dimensional metrizable compactum $K$.
\end{fakt}

In paper \cite{GartFeng} authors introduce the following notion of \emph{fd}-derivative and \emph{fd}-height:
\begin{df}\label{def_fd_pochodna}
Let $X$ be a topological space. For every ordinal $\alpha$ we define using transfinite induction the $\alpha$th $fd$-derivative $X^{[\alpha]}$:
\begin{itemize}
\item[(i)] $X^{[0]} = X$
\item[(ii)]$X^{[\alpha +1]} =X^{[\alpha]} \setminus I(X^{[\alpha]})$, where\\
$I(Y)=\bigcup \{U \subset Y:\; U \textrm{ is open in } Y \textrm{ and finite dimensional}\}$
\item[(iii)]$X^{[\alpha]} = \bigcap_{\beta < \alpha} X^{[\beta]}$ for
$\alpha \in \mathbf{Lim}$.
\end{itemize}
If $X^{[\alpha]} = \emptyset$, for some $\alpha$, then  we also define the $fd$-height of $X$ $$fdh(X)=\min\{\alpha:\; X^{[\alpha]}=\emptyset\}.$$
\end{df}

The following fact can be easily deduced from the Baire category theorem:
\begin{fakt}\label{pochodna_przel_wym}
Let $K$ be a metrizable compactum. Then $K$ is strongly countable dimensional if and only if $fdh(K)<\omega_1$.
\end{fakt}
A useful fact concerning the $fd$-derivative of a subspace is the following:
\begin{fakt} \label{lem_o_pochodnej}
For every subset $A$ of the space $X$, we have $A^{[\alpha]} \subset X^{[\alpha]} \cap
A$.
\end{fakt}

We will also use another known variation of such topological derivative:

\begin{df}\label{c_pochodna}
Let $X$ be a topological space. For every ordinal $\alpha$ we define using transfinite induction the $\alpha$th \emph{c}-derivative $X_c^{[\alpha]}$:
\begin{itemize}
\item[(i)] $X_c^{[0]} = X$
\item[(ii)]$X_c^{[\alpha +1]} =X_c^{[\alpha]} \setminus J(X_c^{[\alpha]})$, where\\
$J(Y)=\bigcup \{U \subset Y:\; U \textrm{ is an open subset of Y with compact closure}\}$
\item[(iii)]$X_c^{[\alpha]} = \bigcap_{\beta < \alpha} X_c^{[\beta]}$ for
$\alpha \in \mathbf{Lim}$.
\end{itemize}
As in the previous case, we also define $c$-height $ch(X)=\min\{\alpha:\; X_c^{[\alpha]}=\emptyset\}$ for $X$ such that $X_c^{[\alpha]} = \emptyset$ for some $\alpha$.
\end{df}

Observe that if a space $X$ is zero-dimensional, then we can replace the set $J(Y)$ in the above definition by
$$J'(Y)=\bigcup \{U \subset Y:\; U \textrm{ is an open and compact subset of Y}\}.$$

Again, using the Baire category theorem we can easily obtain

\begin{fakt}\label{pochodna_c}
If $X$ is a Polish space, then X is $\sigma$-compact if and only if $ch(X)<\omega_1$.
\end{fakt}

For this derivative, we have a weaker counterpart of Fact \ref{lem_o_pochodnej}:

\begin{fakt} \label{lem_o_pochodnej_c}
For every closed subset $A$ of the space $X$, we have $A_c^{[\alpha]} \subset X_c^{[\alpha]} \cap
A$.
\end{fakt}

\section{Preservation of compactness by $u$-domination. Gulko's support approach.}
It was proved by Uspenskii in \cite{U} that if $X$ is pseudocompact and $C_p(X)$ and $C_p(Y)$ are uniformly homeomorphic,
then $Y$ is pseudocompact too.
However, it follows from \cite[Proposition 2]{U} that, in fact,
the same assertion is true if instead of uniform homeomorphism we had a uniformly continuous surjection from
$C_p(X)$ onto $C_p(Y)$ (see Remark \ref{remark} below).

In this section we will present a different proof of Uspenskii's theorem (see Corollary \ref{pseudocompact} below). Our reasoning is an adaptation of a
method used by Baars \cite[Lemma 2.1]{B} and relies on the idea of a support introduced by Gul'ko \cite{G}.
This approach shows a certain similarity between properties of the Gul'ko support and of the
usual notion of support considered in the case of linear continuous maps between
function spaces (see remarks following Theorem \ref{compactness}).

Let $\varphi:C_p(X)\to C_p(Y)$ be a uniformly continuous surjection.
It was proved by Gul'ko \cite{G} (cf. \cite{MP}) that
to each point $y\in Y$, we can assign a (nonempty) finite set
$K(y)\subseteq X$ (called the Gul'ko support of $y$) and a real number $a_y$ such that:
\begin{enumerate}
 \item[(i)] $a_y=\sup\{|\varphi(f)(y)-\varphi(g)(y)|:f,g\in C_p(X)\;\text{and}\;\forall x\in K(y) \;|f(x)-g(x)|<1\}$
 \item[(ii)] If $A\subsetneq K(y)$ then for any $r>0$ there are $f,g\in C_p(X)$ such that $\forall x\in A\; |f(x)-g(x)|<1$ and $|\varphi(f)(y)-\varphi(g)(y)|>r$.
\end{enumerate}

In other words, properties (i) and (ii) say that the set $K(y)$ is the minimal set (with respect to inclusion) such that
$$\sup\{|\varphi(f)(y)-\varphi(g)(y)|\colon f,g\in C_p(X)\;\text{and}\; \forall x\in K(y) \;|f(x)-g(x)|<1\}<\infty.$$
Since $\varphi$ is a surjection, the set $K(y)$ is nonempty for any $y\in Y$.

\begin{lem}\label{lemat}
Let $n\in \N$. If $|f(x)-g(x)|<n$, for each $x\in K(y)$, then $|\varphi(f)(y)-\varphi(g)(y)|\leq n\cdot a_y$
\end{lem}
\begin{proof}
For $k\in\{0,1,\ldots,n\}$, put $f_k=f+\frac{k}{n}(g-f)$. Then  $|f_{k+1}(x)-f_{k}(x)|<1$, for each $x\in K(y)$. By property (i) we have
$|\varphi(f_{k+1})(y)-\varphi(f_{k})(y)|\leq a_y$. This gives
\begin{align*}
|\varphi(f)(y)-\varphi(g)(y)|=
|\varphi(f_0)(y)-\varphi(f_n)(y)|&\leq \\
|\varphi(f_0)(y)-\varphi(f_1)(y)|+\ldots+
|\varphi(f_{n-1})(y)-\varphi(f_n)(y)|&\leq n\cdot a_y
\end{align*}
\end{proof}

Recall that a subset $A$ of a space $X$ is \textit{bounded} if for every $f\in C_p(X)$, the set $f(A)$ is bounded in $\mathbb{R}$.

\begin{fakt}\label{bounded}
 Let $\varphi:C_p(X)\to C_p(Y)$ be a uniformly continuous surjection. If $A\subseteq X$ is bounded, then the set $K_A=\{y\in Y:K(y)\subseteq A\}$
 is bounded too.
\end{fakt}
\begin{proof}
Without loss of generality, we may assume that $\varphi(\underline{0})=\underline{0}$, where $\underline{0}$ is a constant function equal to zero in the respective
space.

Striving for a contradiction, suppose that
there is $g\in C_p(Y)$ such that $g(K_A)$ is not bounded in $\RR$. Without loss of generality, we can assume that there is a discrete set
$\{y_n:n\in \N\}\subseteq K_A$ such that $g(y_n)=n$. Let $h:\RR\to \RR$ be a continuous function satisfying $h(n)=a_{y_n}\cdot n$. Since $\varphi$ is a
surjection, there is $f\in C_p(X)$ with $\varphi(f)=h \circ g$, so $\varphi(f)(y_n)=a_{y_n}\cdot n$. The set $A$ is bounded and thus there is $m\in\N$ such
that $f(A)\subseteq (-m,m)$. Take $n>m$. By Lemma \ref{lemat}, we have
$$n\cdot a_{y_n}=\varphi(f)(y_n)=|\varphi(f)(y_n)-\varphi(\underline{0})(y_n)|\leq a_{y_n}\cdot m< a_{y_n}\cdot n,$$
a contradiction.
\end{proof}

\begin{wn}\label{pseudocompact}
If $\varphi:C_p(X)\to C_p(Y)$ is a uniformly continuous surjection and $X$ is pseudocompact, then $Y$ is pseudocompact.
\end{wn}

\begin{remark}\label{remark}
In \cite[Proposition 2]{U}, Uspenskii proved that a space $X$ is pseudocompact if and only if $C_p(X)$, as a uniform structure, is $\sigma$-precompact.
Now, a uniformly continuous surjection
$\varphi:C_p(X)\to C_p(Y)$ can be extended to a uniformly continuous function $\widehat{\varphi}:\RR^X\to \RR^Y$ (see \cite[8.3.10]{E}).
If $X$ is pseudocompact, then there is a $\sigma$-compact set $Z\subseteq \RR^X$ with $C_p(X)\subseteq Z$. The
set $\widehat{\varphi}(Z)$ is a $\sigma$-compact subset of $\RR^Y$ and since $\varphi$ is a surjection, $C_p(Y)\subseteq \widehat{\varphi}(Z)$. This means
that $C_p(Y)$ is $\sigma$-precompact and
hence by Uspenskii's theorem the space $Y$ is pseudocompact.
\end{remark}

\begin{tw}\label{compactness}
Let $X$ be a compact metrizable space. If $\varphi:C_p(X)\to C_p(Y)$ is a uniformly continuous surjection, then $Y$ is compact metrizable.
\end{tw}
\begin{proof}
Since $X$ is compact metrizable, it has countable networkweight $nw(X)$.
We have $\omega=nw(X)=nw(C_p(X))=nw(C_p(Y))=nw(Y)$ (cf. \cite[Problem 172]{Tk}) and hence $Y$ is Lindel\"of.
By Corollary
\ref{pseudocompact}
the space $Y$ is also pseudocompact so it is compact \cite[Problem 138]{Tk}. Since it has a countable networkweight, it is metrizable.
\end{proof}

Proposition \ref{bounded} reveals that the Gul'ko support $K(y)$ has a similar feature of ``transferring boundedness''
as the usual support $\supp_\varphi(y)$ considered in the case
of linear mappings (see \cite[Lemma 2.1]{B}). In the case of linear surjections Arhangel'skii proved in \cite{A} a theorem in the ``opposite direction'': \textit{If
$\varphi:C_p(X)\to C_p(Y)$ is a linear continuous surjection and $A\subseteq Y$ is bounded, then the set
$\supp_{\varphi}(A)=\bigcup_{y\in A}\supp_{\varphi}(y)$ is bounded too.}
This theorem has very important consequences on the behavior of topological properties under the $l$-equivalence relation (two spaces $X$ and $Y$ are said to
be \textit{$l$-equivalent} provided $C_p(X)$ and $C_p(Y)$ are linearly homeomorphic). This motivates the following question.
\begin{question}
Suppose that $\varphi:C_p(X)\to C_p(Y)$ is a uniformly continuous surjection and let $A\subseteq Y$ be bounded. Is it true that $K(A)=\bigcup_{y\in A}K(y)$ is
bounded?
\end{question}
Even the affirmative answer to the following particular case of the above question may have interesting consequences.
\begin{question}
Suppose that $\varphi:C_p(X)\to C_p(Y)$ is a uniformly continuous surjection and let $A\subseteq Y$ be a convergent sequence.
Is it true that $K(A)=\bigcup_{y\in A}K(y)$ is
bounded?
\end{question}

\section{Uniformly continuous surjections and dimension}

It is known that if $X$ is compact metrizable strongly countable-dimensional (zero-dimensional) and if there is a linear continuous surjection from
$C_p(X)$ onto $C_p(Y)$, then
$Y$ has to be strongly countable-dimensional (zero-di\-men\-sion\-al), cf. \cite{LLP}, \cite{GartFeng}. It is however unclear, whether the same conclusion can be derived for uniformly
continuous surjections, i.e. the following question is open.

\begin{question}\label{pytanie_przel_wym}
Let $X$ be a compact metrizable strongly countable-dimensional (zero-dimensional) space. Suppose that $Y$ is $u$-dominated by $X$. Is $Y$ necessarily
strongly countable-dimensional (zero-dimensional)?
\end{question}

We will prove that the answer to the above question is affirmative provided a uniformly continuous surjection between $C_p(X)$ and $C_p(Y)$ is $c$-good
for some $c>0$, see Definiton \ref{c-good}. We have the following.

\begin{tw}\label{presrvation_count_dim}
Let $X$ be a compact metrizable space. Suppose that $\varphi: C_p(X)\to C_p(Y)$ is a uniformly continuous surjection
which is $c$-good for some $c>0$. Then $Y$ is compact metrizable. Moreover,
\begin{enumerate}[(a)]
 \item if $X$ is zero-dimensional, then $Y$ is zero-dimensional.
 \item If $X$ is strongly countable-dimensional, then $Y$ is strongly countable-dimensional.
\end{enumerate}
\end{tw}
\begin{proof}
It follows from Theorem \ref{compactness} that $Y$ is compact metrizable. Let us prove statements (a) and (b).
For a natural number $q$, denote by $[X]^q$ the space of all $q$-element subsets of $X$ endowed with the Vietoris topology. 
It was shown by Gul'ko \cite{G} (cf. the proof of Proposition 3.1 in \cite{MP}) that
there is a countable family $\{M(p,q):p,q\in \mathbb{N}\}$ of $F_\sigma$
subsets of $Y$ that covers $Y$ and for any $p,q\in \mathbb{N}$ there is a function
$$K_p:M(p,q)\to [X]^q,$$
%\textcolor{red}{Moze dobrze by bylo wspomniec co to $[X]^q$?. Jak sadzisz?}
whose restriction to any $M\subseteq M(p,q)$ closed in $Y$, is perfect (cf. \cite[Property (9)]{MP}) with the following property (see \cite{MP}):
\begin{align}\label{K_p}
&\forall y\in M(p,q)\\ &\sup\{|\varphi(f)(y)-\varphi(g)(y)|:f,g\in C_p(X),\;\| (f-g)\upharpoonright K_p(y) \|<1\}\leq p\nonumber
\end{align}
It follows that for any $p,q\in \mathbb{N}$ we have $M(p,q)=\bigcup_{m\in \mathbb{N}} M_m(p,q)$, where $M_m(p,q)\subseteq Y$ are closed and
for each $m\in \mathbb{N}$ the function $K_p\upharpoonright M_m(p,q)$ is perfect.

Let us fix $p,q$ and $m$. %\textcolor{red}{Tutaj zamiast $M_m(p,q)$ wszedzie bylo $Y$. Zmianielem to.}
We claim that
\begin{align}\label{fin-to-one}
\text{for any}\; z\in M_m(p,q)\; \text{the set}\; \{y\in M_m(p,q): K_p(y)=K_p(z)\}\; \text{is finite.}
\end{align}
Suppose that this is not the case, i.e. for some $z\in M_m(p,q)$
the set $A=\{y\in M_m(p,q): K_p(y)=K_p(z)\}$ is infinite. Since $K_p$ is perfect, the set $A$ is infinite compact metrizable and hence it contains a convergent sequence
$(y_n)_{n\in \mathbb{N}}$ of distinct points. For $n\in \mathbb{N}$, let $g_n:Y\to [0,2p]$ be a continuous function such that
\begin{equation}\label{g_n}
g_n(y)=
  \left\{\begin{aligned}
&   2p\quad &&\text{if}\quad y=y_n\\
&0  &&\text{if}\quad y=y_k,\; \text{for }k\neq n
  \end{aligned}
 \right.
\end{equation}
%\begin{align}\label{g_n}
%g_n(y_n)=3p\quad \text{and}\quad g_n\upharpoonright (Y\setminus U_n)\equiv 0.
%\end{align}
It follows from our assumption that, for each $n\in\mathbb{N}$, there is $f_n\in C_p(X)$ with
$$\varphi(f_n)=g_n \quad \text{and}\quad \| f_n \| \leq c \| g_n \|=2pc.$$
It follows that the set $F=\{f_n:n\in \mathbb{N}\}$ is a subset of the Tychonoff cube $[-2pc,2pc]^X$. Since the latter space is compact, the set $F$
has a complete accumulation point $f\in [-2pc,2pc]^X$.
One can find $f_j,f_k\in F$, $j\neq k$ with $\|(f_j-f)\upharpoonright K_p(z)\|<\frac{1}{2}$ and $\|(f_k-f)\upharpoonright K_p(z)\|<\frac{1}{2}$.
We have $\|(f_j-f_k)\upharpoonright K_p(z)\|<1$
and $K_p(y_j)=K_p(z)$. Thus
$$|g_j(y_j)-g_k(y_j)|=|\varphi(f_j)(y_j)-\varphi(f_k)(y_j)|\leq p,$$
by \eqref{K_p}. On the other hand, we infer from \eqref{g_n} that $|g_j(y_j)-g_k(y_j)|=2p$. The obtained contradiction proves
\eqref{fin-to-one}.

This means that
for any $p, q, m\in\mathbb{N}$, the closed mapping $K_p\upharpoonright M_m(p,q):M_m(p,q)\to [X]^q$ is finite-to-one. In particular it has
zero-dimensional fibers.

To prove part (a) of the theorem, assume that $X$ is zero-dimensional. Then $[X]^q$ is zero-dimensional. It follows that $M_m(p,q)$ is zero-dimensional. Indeed,
otherwise $M_m(p,q)$ would contain a nontrivial connected subset $C$ (see \cite[6.2.9]{E}). Since $[X]^q$ is zero-dimensional the image of $C$ under the mapping
$K_p\upharpoonright M_m(p,q)$ is a point which is not possible because $K_p\upharpoonright M_m(p,q)$ is finite-to-one.
Now, $Y=\bigcup\{M_m(p,q):m,p,q\in\mathbb{N}\}$ is zero-dimensional being a countable union of closed zero-dimensional spaces (see \cite[1.3.1]{Eng}).

To show (b), assume that $X$ is strongly countable-dimensional. Then $[X]^q$ is also strongly countable-dimensional.
Since the mapping $K_p\upharpoonright M_m(p,q)$ has zero-dimensional fibers,
it follows from \cite[5.4.7]{Eng} that $M_m(p,q)$ is strongly countable-dimensional. Since $Y=\bigcup\{M_m(p,q):m,p,q\in\mathbb{N}\}$
and each set $M_m(p,q)$ is closed in $Y$, we conclude that $Y$ is strongly countable-dimensional.
\end{proof}

\section{On spaces $u$-dominated by the unit interval}

 In this section we will prove the main result of the first part of paper (see Definition \ref{def_nierownosci}):

 \begin{tw}\label{charakteryzacja_przestrzeni_dominowanych_przez_odcinek}
 For a space $X$, we have $C_p(X)\bumpineq  C_p(\I)$ if and only if $X$ is compact, metrizable and strongly countable-dimensional.
 \end{tw}

The ``only if'' part of this theorem follows easily from Theorem \ref{presrvation_count_dim}. To prove the reverse implication it suffices to use Fact \ref{pochodna_przel_wym} and the following

\begin{tw}\label{w jedna strone}
For every compact metrizable space $K$ and every ordinal $\alpha$ we have $C_p(K,K^{[\alpha]}) \bumpineq  C_p(\I)$.
\end{tw}

To prove this theorem we need the following :
\begin{lem}\label{lemat_o_rozkladzie}
Let $\{U_i\}_{i \in \N}$ be a collection of open subsets of $K$ such that $\forall i \in \N \; \cl U_{i+1} \subset U_i$ and $U_0=K$. Then
$$ C_p(K,K_0) \bumpineq
\Pi^{*}_{i \in \N } C_p(\cl U_i \setminus U_{i + 1}) \times \Pi^*_{i \in
\N } C_p(\bd U_i)$$
where $K_0=\bigcap U_i$.
\end{lem}

\begin{proof}
Let us observe that $K= K_0\cup\bigcup_{i \in \N} \cl U_i \setminus U_{i + 1}$. Using Fact \ref{c_0_iloczyny}, Proposition \ref{wn_Dug}, and Corollary \ref{wn_wn_Dug} we obtain:

\begin{align*}
&C_p(K,K_0) \bumpeq\\&\bumpeq C_p(K_0\cup\bigcup_{i \in \N} \cl U_i \setminus U_{i + 1},K_0 \cup \bigcup_{i \in \N} \bd U_i) \times C_p(K_0\cup\bigcup_{i \in \N} \bd U_i, K_0)\\
&\equiv \Pi^{*}_{i \in \N } C_p(\cl U_i \setminus U_{i + 1}
, \bd U_{i+1} \cup\bd U_i) \times \Pi^*_{i \in
\N } C_p(\bd U_i)\\ &\bumpineq \Pi^*_{i \in
\N } C_p(\cl U_i \setminus U_{i + 1}) \times \Pi^*_{i \in
\N } C_p(\bd U_i)\,.\hspace{32mm} \qedhere
\end{align*}
\end{proof}

Another important lemma that will be used in the proof is a particular case of Theorem \ref{w jedna strone}:
\begin{lem}\label{fakt_o_pierwszej_pochodnej}
For every compact metrizable space $K$, $C_p(K,K^{[1]}) \bumpineq  C_p(\I)$.
\end{lem}
\begin{proof}
This lemma is an easy consequence of Lemma \ref{lemat_o_rozkladzie}. Indeed, take $\{U_i\}_{i \in \N}$ a collection of open subsets of $K$ such that $\cl U_i \subset U_{i+1}$ and $K^{[1]}=\bigcap_{n \in \N} U_i$. Then by Lemma \ref{lemat_o_rozkladzie}
$$ C_p(K,K^{[1]}) \bumpineq
\Pi^{*}_{i \in \N } C_p(\cl U_i \setminus U_{i + 1}) \times \Pi^*_{i \in
\N } C_p(\bd U_i).$$
However, it is easy to observe that, for every $i \in \N$, the sets $\cl U_i \setminus U_{i + 1}$ and $\bd U_i$ are finitely dimensional, being compact subsets of $I(K)$, see Definition \ref{def_fd_pochodna}. Hence, by Facts \ref{c_0_iloczyny}, \ref{prod_Cp(I)}, and \ref{dominacja_kostki} we obtain
\begin{displaymath}
C_p(K,K^{[1]}) \bumpineq
\Pi^{*}_{i \in \N } C_p(\I) \bumpeq C_p(\I)\,. \qedhere
\end{displaymath}
\end{proof}

\begin{proof}[Proof of Theorem \ref{w jedna strone}]

The theorem will be proved by transfinite induction. The case $\alpha=0$ is obvious. For $\alpha = 1$ the thesis follows from Lemma \ref{fakt_o_pierwszej_pochodnej}. If $\alpha=\beta+1$ then by Proposition \ref{wn_Dug}
$$C_p(K,K^{[\alpha]}) \bumpeq C_p(K,K^{[\beta]}) \times C_p(K^{[\beta]},K^{[\alpha]}).$$
However, from the inductive assumption we obtain that $C_p(K,K^{[\beta]}) \bumpineq  C_p(\I)$ and by Lemma \ref{fakt_o_pierwszej_pochodnej} we have
$C_p(K^{[\beta]},K^{[\alpha]}) \bumpineq C_p(\I)$. Since $C_p(\I)\bumpeq C_p(\I) \times C_p(\I)$ (from Fact \ref{prod_Cp(I)}) the proof of inequality $C_p(K,K^{[\alpha]}) \bumpineq C_p(\I)$ is finished.

Let us assume that $\alpha$ is a limit ordinal. We can assume that $\alpha < \omega_1$ because $K$ has a countable base and therefore the $fd$-derivative stabilizes after some countable ordinal as it happens in case of standard Cantor-Benedixson derivative.
Let $K^{[\alpha]}=\bigcap_{i \in \N}K^{[\beta_i]}$ where $\beta_i$ is an increasing sequence of ordinals converging to $\alpha$.
Consider  a collection $\{U_i\}_{i \in \N}$ of open subsets of $K$ such that:
\begin{itemize}
\item[(i)] $U_0 = K$;
\item[(ii)] $U_i \supset K^{[\beta_i]}$;
\item[(iii)] $U_i \supset \rm{cl}U_{i+1}$;
\item[(iv)] $\bigcap^{\infty}_{i = 0} U_i = K^{[\alpha]}$.
\end{itemize}
From Lemma \ref{lemat_o_rozkladzie} we obtain:
$$ C_p(K,K^{[\alpha]}) \bumpineq
\Pi^{*}_{i \in \N } C_p(\cl U_i \setminus U_{i + 1}) \times \Pi^*_{i \in
\N } C_p(\bd U_i).$$
Applying Fact \ref{lem_o_pochodnej}: $$(\bd U_i)^{[\beta_{i+1}]} \subset (\cl U_i \setminus U_{i+1})^{[\beta_{i+1}]} \subset K^{[\beta_{i+1}]} \cap (\cl U_i \setminus U_{i+1}) = \emptyset.$$ Hence, by inductive assumption $C_p(\cl U_i \setminus U_{i + 1}) \bumpineq C_p(\I)$ and $C_p(\bd U_i) \bumpineq C_p(\I)$. Facts \ref{c_0_iloczyny}, \ref{prod_Cp(I)}, and \ref{dominacja_kostki} give us the desired conclusion
\[
C_p(K,K^{[\alpha]}) \bumpineq
\Pi^{*}_{i \in \N } C_p(\I) \bumpeq C_p(\I). \qedhere
\]
\end{proof}

From Theorem \ref{charakteryzacja_przestrzeni_dominowanych_przez_odcinek} we immediately obtain

\begin{wn}
Every strongly countable-dimensional compact, metrizable space $K$ is $u$-dominated by the unit interval.
\end{wn}

\section{Uniform homeomorphisms between $C_p(X)$ and $C_p(X)\times C_p(X)$}

In this section we will prove the following theorem which partially answers Question 5.8 from \cite{KM}.

\begin{tw}\label{tw_kwadrat}
Let $X$ be an infinite Polish zero-dimensional space. Then we have $C_p(X)\bumpeq C_p(X) \times C_p(X)$, in particular the spaces $C_p(X)$ and $C_p(X) \times C_p(X)$ are uniformly homeomorphic.
\end{tw}

Let us point out that the assumption on zero-dimensionality of $X$ cannot be dropped in the above theorem, since van Mill, Pelant, and Pol \cite{vMPP}
gave an example of a one-dimensional compact metrizable space $M$ such that $C_p(M)$ is not uniformly homeomorphic to its square. Also, the assumption that
$X$ is completely metrizable cannot be removed, because Krupski and Marciszewski proved in \cite{KM} that
there exists a zero-dimensional subspace $B$ of the real line with $C_p(B)$ not homeomorphic to its square. We do not know if Theorem \ref{tw_kwadrat}
holds true for Borel
subspaces $X$ of the Cantor set.

\begin{question}
Let $X$ be a Borel subspace of the Cantor set. Are the spaces $C_p(X)$ and $C_p(X) \times C_p(X)$ (linearly, uniformly) homeomorphic?
In particular, what if $X$ is $\sigma$-compact?
\end{question}

Recall that a space $X$ is \textit{scattered} if no nonempty subset $A\subseteq X$ is dense-in-itself.

The following results for linear homeomorphisms were proved by Baars and de Groot (part (a) and (b)) \cite{BG} and Arhangel'skii (part (c)) \cite{A2}. Our version requires basically the same arguments, but for the reader convenience we include a short justification.
\begin{prop}\label{Canor_wymierne_niewym}
Let $X$ be a zero-dimensional metrizable space.
\begin{itemize}
\item[(a)] if $X$ is uncountable compact, then $C_p(X)\bumpeq C_p(2^\omega)$,
\item[(b)] if $X$ is countable and not scattered, then $C_p(X)\bumpeq C_p(\mathbb{Q})$,
\item[(c)]
if $X$ is Polish and not $\sigma$-compact, then $C_p(X)\bumpeq C_p(\omega^\omega)$.
\end{itemize}
\end{prop}

\begin{proof} (a)
First, let us observe that $C_p(2^\omega) \bumpeq \Pi^*_{i \in \N } C_p(2^\omega)$. Indeed, from the topological characterization of the Cantor set it  follows  that it is homeomorphic to the one point compactification $K$ of the space $\mathbb{N}\times 2^\omega$. Let $\infty$ denote the point at infinity of this compactification. By Corollary \ref{o_punkcie} we have
$$C_p(2^\omega) \bumpeq C_p(K) \bumpeq C_p(K,\infty) \bumpeq \Pi^*_{i \in \N } C_p(2^\omega).$$
It is well-known that $X$ contains a closed copy $A$ of $2^\omega$ and $2^\omega$ contains a closed copy $B$ of $X$. Therefore we can apply Corollary \ref{wn_wn_Dug} and one of the versions of Decomposition Scheme:
\begin{eqnarray*}
C_p(2^\omega)  &\bumpeq& \Pi^*_{i \in \N } C_p(2^\omega)  \bumpeq \Pi^*_{i \in \N } [C_p(2^\omega,B)\times C_p(B)]\\
&\bumpeq& \Pi^*_{i \in \N } [C_p(2^\omega,B)\times C_p(X,A)\times C_p(A)]\\
&\bumpeq& \Pi^*_{i \in \N } C_p(2^\omega,B)\times \Pi^*_{i \in \N } C_p(X,A) \times \Pi^*_{i \in \N } C_p(A)\\
&\bumpeq& C_p(X,A) \times \Pi^*_{i \in \N } C_p(2^\omega,B)\times \Pi^*_{i \in \N } C_p(X,A) \times \Pi^*_{i \in \N } C_p(A)\\
&\bumpeq& C_p(X,A) \times \Pi^*_{i \in \N } [C_p(2^\omega,B)\times C_p(X,A)\times C_p(A)]\\
&\bumpeq& C_p(X,A)\times \Pi^*_{i \in \N } C_p(2^\omega) \bumpeq C_p(X,A)\times C_p(2^\omega)\\
&\bumpeq& C_p(X,A)\times C_p(A) \bumpeq C_p(X)\,.
\end{eqnarray*}
The proof of parts (b) and (c) are very similar. Here, we can use countable products instead of $ \Pi^*$-products since, for $Y=\mathbb{Q},\omega^\omega$, the space $Y$ is homeomorphic to $\omega\times Y$, hence we have $C_p(Y) \bumpeq (C_p(Y))^\omega$.
\end{proof}

Gul'ko proved in \cite{Gu1} that all infinite countable compact spaces are $u$-equivalent. However, inspecting his proof one can verify that he actually proved the following

\begin{tw}
For every infinite countable compact spaces $X$ and $Y$, we have $C_p(X)\bumpeq C_p(Y)$.
\end{tw}

From this theorem and Proposition \ref{Canor_wymierne_niewym} immediately follows

\begin{wn}\label{square_cpt}
If $X$ is an infinite zero-dimensional, metrizable compact space, then $C_p(X)\bumpeq C_p(X)\times C_p(X)$.
\end{wn}

\begin{proof}[Proof of Theorem \ref{tw_kwadrat}]\hfill

{\bf Case 1.} If $X$ is not $\sigma$-compact then the desired conclusion follows easily from Proposition \ref{Canor_wymierne_niewym}(c).

{\bf Case 2.} If $X$ is $\sigma$-compact then it has countable $c$-height, see Fact \ref{pochodna_c}. We will prove our assertion by induction on $ch(X)$.

We start with $X$ of $c$-height 1 ($X$ is nonempty), i.e., a locally compact $X$. If $X$ is compact then we can apply Corollary \ref{square_cpt}. In the opposite case, since $X$ is zero-dimensional,  we can cover it by a family $\{U_n: n\in\N\}$ of pairwise disjoint nonempty open and compact subsets. Some of sets $U_n$ can be finite, and we need to consider three cases. If infinitely many $U_n$ are infinite, then without loss of generality we can assume all of them are infinite (if necessary, we can assign, in a one-to-one way, to each finite $U_n$ an infinite $U_{k(n)}$, and replace these two sets by their union). Then by Fact \ref{iloczyny} and Corollary \ref{square_cpt} we have
\begin{eqnarray*}
C_p(X)  &\bumpeq& \Pi_{i \in \N } C_p(U_n) \bumpeq
\Pi_{i \in \N } [C_p(U_n)\times C_p(U_n)]\\
 &\bumpeq& \Pi_{i \in \N } C_p(U_n)\times \Pi_{i \in \N } C_p(U_n) \bumpeq C_p(X)\times C_p(X)\,.
\end{eqnarray*}
If only finitely many $U_n$ are infinite and $X$ is not discrete, then $X$ is homeomorphic to the discrete union of $\N$ and some infinite compact space $Y$. Then again using Fact \ref{iloczyny} and Corollary \ref{square_cpt} we obtain
\begin{eqnarray*}
C_p(X)  &\bumpeq& C_p(\N)\times C_p(Y) \bumpeq \mathbb{R}^\N\times C_p(Y)\\
&\bumpeq& \mathbb{R}^\N\times C_p(Y)\times \mathbb{R}^\N\times C_p(Y)  \bumpeq C_p(X)\times C_p(X)\,.
\end{eqnarray*}
Finally, if $X$ is discrete then $C_p(X)= \mathbb{R}^X$ and our conclusion is trivial.

Assume now that $ch(X)=\alpha>1$ and the hypothesis of the theorem holds true for all spaces $Y$ with $ch(Y)< \alpha$. Let us begin with more complex case when $\alpha=\beta+1$, for some countable ordinal $\beta$. Then the subspace $X_c^{[\beta]}$ is locally compact.

First, we will consider the case when $X_c^{[\beta]}$ is  compact. Fix some admissible metric $d$ on $X$.
Since $X$ is zero-dimensional we can choose a sequence of clopen subsets $W_n$ of $X$, $n\in \mathbb{N}$, such that
\begin{itemize}
\item[(i)] $W_0 = X$,
\item[(ii)] $W_{n+1}\subseteq W_n$ for $n\in \mathbb{N}$,
\item[(iii)] $X_c^{[\beta]}\subseteq W_n\subseteq \{x\in X: d(x,X_c^{[\beta]})<1/n\}$ for $n\ge 1$.
\end{itemize}
By condition (iii) we have $\bigcap_{n\in \mathbb{N}}W_n = X_c^{[\beta]}$. For every $n\in \mathbb{N}$ put
$$V_n = W_n \setminus W_{n+1}\,.$$
The sets $V_n$ are clopen in $X$, hence they are zero-dimensional, $\sigma$-compact and Polish. If, for some $k\in \mathbb{N}$ and all $n\ge k$, the sets $V_n$ were finite, then the set $W_k$ would be compact and open in $X$, hence disjoint from $X_c^{[1]}$, a contradiction with condition (iii) and the inequality $\beta\ge 1$. Therefore, $V_n$ are infinite for infinitely many $n$. Without loss of generality we can assume that this is the case for all $n$ (if necessary, we can glue any finite $V_n$ with the first infinite $V_k$ with a greater index). Since $V_n$ are disjoint from $X_c^{[\beta]}$, we have $ch(V_n)\le\beta$, for $n\in \mathbb{N}$, see Fact \ref{lem_o_pochodnej_c}. Therefore, we can apply our inductive assumption for all $V_n$. By Facts \ref{iloczyny} and \ref{c_0_iloczyny}, and Corollaries \ref{wn_wn_Dug} and \ref{square_cpt} we have
\begin{eqnarray*}
C_p(X)  &\bumpeq& C_p(X_c^{[\beta]})\times C_p(X,X_c^{[\beta]}) \bumpeq C_p(X_c^{[\beta]})\times
\Pi^*_{i \in \N } C_p(V_n)\\
&\bumpeq& C_p(X_c^{[\beta]})\times C_p(X_c^{[\beta]}) \times
\Pi^*_{i \in \N } [C_p(V_n)\times C_p(V_n)]\\
&\bumpeq& C_p(X_c^{[\beta]})\times \Pi^*_{i \in \N } C_p(V_n)\times C_p(X_c^{[\beta]}) \times \Pi^*_{i \in \N } C_p(V_n)\\
&\bumpeq& C_p(X_c^{[\beta]})\times C_p(X,X_c^{[\beta]})\times C_p(X_c^{[\beta]})\times C_p(X,X_c^{[\beta]})\\
&\bumpeq& C_p(X)\times C_p(X)\,.
\end{eqnarray*}

If $X_c^{[\beta]}$ is not compact, since it is
zero-dimensional and locally compact,  we can cover it by a family $\{U_n: n\in \N\}$ of pairwise disjoint nonempty relatively open and compact subsets of $X_c^{[\beta]}$. Since $X_c^{[\beta]}$ is closed in $X$, for every $n$, the union $A_n=\bigcup_{i\ne n} U_i$ is closed in $X$. Let $\mathcal{V}$ be a cover of $X$ consisting of pairwise disjoint clopen sets and inscribed into the open cover $\{X\setminus A_n: n\in\N\}$. Observe that any element of $\mathcal{V}$ can intersect at most one set $U_n$, and, by compactness, any $U_n$ is covered by finitely many elements of $\mathcal{V}$. Therefore we can divide $\mathcal{V}$ into pairwise disjoint finite subfamilies  $\mathcal{V}_n$, $n\in\N$, such that the union $V_n=\bigcup \mathcal{V}_n$ covers $U_n$. Then the family $\{V_n: n\in \N\}$ is a cover
of $X$ consisting of pairwise disjoint clopen sets such that $U_n=V_n\cap X_c^{[\beta]}$, for $n\in\N$.
Clearly, for each $n\in\N$ we have $(V_n)_c^{[\beta]}= U_n$, hence, by previous case, we have $C_p(V_n)\bumpeq C_p(V_n)\times C_p(V_n)$. It remains to observe that we can identify $C_p(X)$ with $\Pi_{i \in \N } C_p(V_n)$ and use Fact \ref{iloczyny}.

Finally, we consider the case when $\alpha=ch(X)$ is a limit ordinal. Let $\{U_n: n\in \N\}$ be a cover of $X$ consisting of pairwise disjoint clopen sets and inscribed into the open cover $\{X\setminus X_c^{[\beta]}: \beta<\alpha\}$. For each $U_n$ we have $ch(U_n)<\alpha$, see Fact \ref{lem_o_pochodnej_c}. Since $ch(X)=\alpha$, infinitely many $U_n$ must be infinite. Using similar argument as in previous cases we can assume that all $U_n$ are infinite.
Therefore, we can identify $C_p(X)$ with $\Pi_{i \in \N } C_p(U_n)$ and use our inductive assumption for each $C_p(U_n)$.
\end{proof}

The next corollary gives an affirmative answer to a part of Question 5.7 from \cite{KM}.

\begin{wn}\label{square_countable}
Let $X$ be an infinite countable metrizable space. Then we have $C_p(X)\bumpeq C_p(X) \times C_p(X)$, in particular the spaces $C_p(X)$ and $C_p(X) \times C_p(X)$ are uniformly homeomorphic.
\end{wn}

\begin{proof}
$X$ is either scattered, hence completely metrizable, and we can apply Theorem \ref{tw_kwadrat}, or $X$ is not scattered and Proposition \ref{Canor_wymierne_niewym}(b) applies.
\end{proof}

From Theorem \ref{tw_kwadrat}, the above corollary, and another version of Decomposition Scheme immediately follows

\begin{wn}
Let $X_0,X_1$ be zero-dimensional Polish (countable metrizable) spaces. Then $C_p(X_0)$ is uniformly homeomorphic to $C_p(X_1)$ provided $C_p(X_i)$ is uniformly homeomorphic to $C_p(X_{i-1})\times E_i$, for some topological vector spaces $E_i$, $i=0,1$.
\end{wn}


\begin{thebibliography}{99}


\bibitem{A} A.V.\ Arhangel'skii, {\em On linear homeomorphism of function spaces}, Soviet Math. Dokl. 25 (1982), 852--855.

\bibitem{A2} A.V.\ Arhangel’skii, {\em Linear topological classification of spaces of continuous functions in the topology of pointwise convergence}
 (Russian) Mat. Sb. 181 (1990), no. 5, 705--718; translation in Math. USSR-Sb. 70 (1991), no. 1, 129--142.

\bibitem{B} J.\ Baars, {\em A note on linear mappings between function spaces}, Comment. Math. Univ. Carolin. 34 (1993), no. 4, 711--715.

\bibitem{BG} J.A.\ Baars, J.A.\ de\ Groot, {\em On topological and linear equivalence of certain function spaces},
CWI Tract 86, Stichting Mathematisch Centrum, Centrum voor Wiskunde en Informatica, Amsterdam, 1992.

\bibitem{Be} G.\ Beer, {\em Norms with infinite values}, J. Convex Anal. 22 (2015), no. 1, 37--60.

\bibitem{E} R.\ Engelking, {\em General Topology}, Sigma Series in Pure Mathematics, 6. Heldermann Verlag, Berlin, 1989

\bibitem{Eng} R.\ Engelking,  {\em Theory of Dimension, finite and
infinite ,} Heldermann Verlag, (1995)

\bibitem{GartFeng} P.\ Gartside and Z.\ Feng, {\em Spaces $\ell$-dominated by $I$ or $\mathbb{R}$}, Topology Appl. 219 (2017), 1--8.


\bibitem{G} S.P.\ Gul'ko, {\em On uniform homeomorphisms of spaces of continuous functions (Russian)},
Trudy Mat. Inst. Steklov. 193 (1992), 82--88; English translation in: Proc. Steklov Inst. Math. 1993, no. 3 (193), 87--93.

\bibitem{Gu1}  S.P.\ Gul'ko, {\em The space $C_p(X)$ for countable
infinite compact $X$ is uniformly homeomorphic to $c_0$}, Bull. Acad. Pol. Sci. 36 (1988), 391--396.

\bibitem{Gor}  R.\ G\'orak, {\em Spaces u-equivalent to the n-cube}, Topology Appl., 132 (2003), 17--27

\bibitem{KM} M.\ Krupski and W.\ Marciszewski, {\em A metrizable $X$ with $C_p (X)$ not homeomorphic to $C_ p (X)\times C_p (X)$},
Israel J. Math. 214 (1) 2016, 245--258.

\bibitem{LLP}
A.\ Leiderman, M.\ Levin, V.\ Pestov, 
\textit{On linear continuous open surjections of the spaces $C_p(X)$}, Topology Appl. 81 (1997), no. 3, 269--279.

\bibitem{LMP}
A.\ Leiderman, S.A.\ Morris, V.\ Pestov, \textit{The free abelian topological group and the free locally convex space on the unit interval},
J. London Math. Soc. (2) 56 (1997), no. 3, 529--538.

\bibitem{Levin}
M.\ Levin, \textit{A property of $C_p[0,1]$}, Trans. Amer. Math. Soc. 363 (2011), no. 5, 2295--2304.

\bibitem{Marciszewski-artykul-przegladowy}
W.\ Marciszewski, \textit{Function Spaces}, in: Recent Progress in General Topology II, M.\ Hu\v{s}ek and J.\ van\ Mill (eds.),
Elsevier 2002, 345-369.

\bibitem{MP} W.\ Marciszewski, J. Pelant, {\em Absolute Borel sets and function spaces},  Trans. Amer. Math. Soc. 349 (1997), no. 9, 3585--3596.

\bibitem{vMPP} J.\ van\ Mill, J.\ Pelant, R.\ Pol, \textit{Note on function spaces with the topology of pointwise convergence},
Arch. Math. (Basel) 80 (2003), no. 6, 655--663.

\bibitem{Pestov}
V. Pestov, {\em The coincidence of the dimension dim of $\ell$-equivalent topological spaces}, Sov. Math. Dokl. 26 (1982) 380--383.

\bibitem{U}V.V.\ Uspenskii, {\em A characterization of compactness in terms of uniform structure in a function space}, Uspekhi Matem. Nauk 37 (1982), 183--184.

\bibitem{Tk}V.V.\ Tkachuk, {\em A $C_p$-theory problem book. Topological and function spaces.} Problem Books in Mathematics. Springer, New York, 2011.

\end{thebibliography}
\end{document}